\documentclass[12pt,reqno]{amsart} 
\usepackage{amsfonts, amssymb,stmaryrd}
\usepackage{latexsym}
\usepackage{graphicx}
\usepackage{xcolor}
\usepackage{hyperref}

\newtheorem{theorem}{Theorem}[section]
\newtheorem{corollary}[theorem]{Corollary}
\newtheorem{lemma}[theorem]{Lemma}
\newtheorem{proposition}[theorem]{Proposition}

\newtheorem{problem}[theorem]{Problem}

\theoremstyle{definition}
\newtheorem{definition}[theorem]{Definition}
\newtheorem{example}[theorem]{Example}

\theoremstyle{remark}

\numberwithin{equation}{section}

\begin{document}
%

\title[ A Conjugacy Invariant from Independence Entropy]{Extracting Invariants of Conjugacy from Independence Entropy}
\author{Bashir Abu Khalil}
\maketitle

\begin{abstract}
The concept of independence entropy for symbolic dynamical systems was introduced in \cite{LMP13}. This notion of entropy measures the extent to which one can freely insert symbols in positions without violating the constraints defined by the shift space. Independence entropy is not invariant under topological conjugacy. We define an invariant version of the independence entropy by setting $ \sup h_{ind} (X) = \sup\{h_{ind}(Y)|Y\simeq X\}.$ This invariant is bounded above by the topological entropy. We prove that equality $ \sup h_{ind} (X)=h(X)$ holds for all sofic shift spaces over $\mathbb{Z},$ then we give an example showing that equality does not hold for general shift spaces.
\end{abstract}

\section{Introduction} \label{sec:intro}
We are concerned with invariants of shift spaces ${(X,\sigma)}\subseteq(\mathbf{A}^{\mathbb{Z}^d},\sigma),$ where $\mathbf{A}$ is a finite alphabet and $\sigma$ is the shift. Topological entropy is the most fundamental numerical invariant associated to a $d$-dimensional shift space. When $d= 1$ and $X$ is either of finite type (SFT), or a sofic shift, the topological entropy is easy to compute as the $\log$ of the largest eigenvalue of a non-negative integer matrix. However, when $d = 2$ there is no known explicit expression for the topological entropy of SFTs,  even for the simplest non-trivial examples. The notion of independence entropy of a shift space defined in \cite{LMP13} comes to measure the part of the entropy resulting from inter-symbol independence in elements of the shift space. The notion of independence entropy is not invariant under topological conjugacy. Still we can define an invariant notion:  \begin{center}
 $ \sup h_{ind} (X) = \sup\{h_{ind}(Y)|Y\simeq X\},$ \end{center} which is bounded above by the topological entropy.
 
  The aim of this paper is to prove the following theorems:
\begin{theorem} \label{thm:main1}
Let $X$ be a one-dimensional sofic shift over an alphabet $\mathbf{A}$. Then $ \sup h_{ind} (X)=h(X)$. 
\end{theorem}
\begin{theorem} \label{thm:main2}
There exists a one-dimensional shift space $X$ with $h(X)>0$ and $ \sup h_{ind} (X)=0$. 
\end{theorem}
This paper is organized as follows. In sections \ref{sec2} and \ref{sec3}, we introduce
the basic theory of shift spaces, and the notion of topological entropy, independence entropy, and supremum independence entropy. In section \ref{sec4}, we show that for any one-dimensional sofic shift space, the supremum independence entropy and topological entropy coincide. In section \ref{sec5}, we give an example of a one-dimensional shift space $X,$ with $ \sup h_{ind} (X) \neq h(X).$

\section{Definitions and Preliminary
Results}\label{sec2}
We start with some definitions and results from symbolic dynamics and graph theory.

An alphabet $\mathbf{A}$ consists of a finite set of discrete symbols such as $\{0,1,2,...,9\}$ or $\{a,b,c,...,z\}.$ The fundamental elements of study in symbolic dynamics are bi-infinite sequences of these letters, drawn from a given alphabet. Such a sequence is denoted by $x=\left( x_{i}\right)_{i\in \mathbb{Z}},$ or by $x=...x_{-2}x_{-1}.x_{0}x_{1}x_{2}...$ where each $x_{i}\in \mathbf{A}.$ The full\emph{\ }$\mathbf{A}$-shift\emph{\ }is the set of all bi-infinite sequences of symbols from $\mathbf{A}\emph{,}$ written as $\mathbf{A^\mathbb{Z}}=\{x=\left( x_{i}\right) _{i\in \mathbb{Z}}\vert \,x_{i}\in \mathbf{A},\forall i\in \mathbb{Z}\}.$

\emph{A block\ } (or\emph{ word}) over $\mathbf{A}$ is a finite sequence $u$ of symbols from $\mathbf{A}\emph{.}$\emph{\ }%
An $n$-block is a block of length $n\emph{,}$ the length denoted by\emph{\ }$\left\vert u\right\vert .$\emph{\ }An\emph{\ }\textit{empty\ block}\emph{\ }is the block with no symbol denoted by $\varepsilon $ with $\left\vert \varepsilon \right\vert =0. $ The set of all \emph{$n$-blocks} over $\mathbf{A}$ is denoted by $\mathbf{A}^{n}.$ A \emph{sub-block} or \emph{sub-word} of $u=a_{1}a_{2}...a_{n}$ is a block of the form $a_{i}a_{i+1}...a_{j}$ where $1\leq i\leq j\leq n.$ The empty block is a sub-block of every block. If $x$ is a point in $\mathbf{A}^{\mathbb{Z}}$ and $i\leq j,$ then $x_{\left[ i,j\right] }=x_{i}x_{i+1}...x_{j}$ will denote the corresponding sub-block. By extension, we will use the notation $x_{[i,\infty)}$ for the \emph{right-infinite word} $x_{i}x_{i+1}x_{i+2}\cdots ,$ similarly $x_{(-\infty,i]}=\cdots x_{i-2}x_{i-1}x_{i}.$ 

 \emph{The shift map }$\sigma $\emph{\ }on the full shift\emph{\ }$\mathbf{A}^{\mathbb{Z}}$\emph{\ }maps a point\emph{\ }$x$\emph{\ }to the point\emph{\ }$y=\sigma(x)$\emph{\ }such that\emph{\ }$y_{i}=x_{i+1},\forall i\in \mathbb{Z} \emph{.}$\emph{\ }Thus every letter is shifted one place to the left. The inverse shift map $\sigma ^{-1}$ shifts every letter to the right. We write $\sigma ^{n}$ to denote the $n$-fold composition. A set of infinite words over $\mathbf{A}$ is a shift space (or sub-shift) if it is closed with respect to the natural product topology of $\mathbf{A}^{\mathbb{Z}}$ and invariant under the shift operator.  A shift space $X$ is sometimes denoted as $\left( X,\sigma _{X}\right)$ to emphasize the role of the shift operator.

For $X \subseteq \mathbf{A}^{\mathbb{Z}}$, let $B_{n}(X)$ denote the set of all $n$-blocks that occur in any of the sequences inside $X.$
 \emph{The language\ of }$X$\emph{\ }is the set:\emph{\ }$B(X)=
\underset{n=0}{\overset{\infty }{\bigcup }}B_{n}(X)$ that\emph{\ }is, all possible "allowed" blocks in all points of\emph{\ }$X\emph{.}$\ A shift space $X$\emph{\ is irreducible }if for every ordered pair  $u,v\in B(X)$ there is a block $w$ such that $uwv\in B(X).$ 

We say that a map $\phi :X\rightarrow Y$ is shift commuting if $\sigma _{Y}\circ\phi =\phi \circ \sigma _{X}$.  Let $X$\ be a shift space and $\mathbf{A},\mathbf{A}^{\prime }$\ two alphabets. Fix $n,m\in \mathbb{Z}$ with $-m\leq n.$ A \emph{block map} or an $(m+n+1)-$\emph{block\ map} is a function $\Phi :B_{m+n+1}(X)\rightarrow \mathbf{A}^{\prime }.$ \emph{The sliding block code }induced by $\Phi $ is the map $\phi :X\rightarrow \left( \mathbf{A}^{\prime}\right) ^{\mathbb{Z}}$ satisfying $\left(\phi \left( x\right)\right) _{i}=\Phi \left(
x_{i-m}...x_{i+n}\right)$ for all $x\in X$ and $i\in \mathbb{Z}.$ If $\phi :X\rightarrow Y$ is a sliding block code, then $\phi \circ \sigma _{X}=\sigma _{Y}\circ \phi ,$ (\cite{LM1}, Proposition 1.5.7). An injective sliding block code $\phi :X\rightarrow Y$ is called an  \emph{embedding} of $X$ into $Y$ and we say that $X$ embeds into $Y$ if such a map exists. A surjective sliding block code $\phi :X\rightarrow Y$ is called a \emph{factor code} from $X$ onto $Y$ and we say that $Y$ is a factor of $X$ if there is a factor code from $X$ to $Y.$ Every bijective sliding block code $\phi :X\rightarrow Y$
has a sliding block inverse (\cite{LM1}, Theorem 1.5.14). In this case we call $\phi $ a \textit{conjugacy} from $X$ to $Y.$ Two shift spaces are \textit{conjugate}\ if there exists a conjugacy between them.

Let $X$ be a shift space over the alphabet $\mathbf{A},$ and $\mathbf{A}_{X}^{\left[ N\right] }=B_{N}(X)$ be the collection of all allowed $N$-blocks in $X.$ We can consider $\mathbf{A}_{X}^{\left[ N\right] }$ as an alphabet in its own right, and form the full
shift $(\mathbf{A}_{X}^{\left[ N\right] })^{\mathbb{Z}}.$  The $\emph{Nth}$ higher block code $\beta_{N}:X\rightarrow (\mathbf{A}_{X}^{\left[ N\right] })^{\mathbb{Z}}$ is defined by $\left( \beta _{N}(x)\right) _{i}=x_{\left[ i,i+N-1\right] }.$  The $N$th higher block shift $X^{\left[ N\right] }$ of $X$ is the image $X^{\left[ N\right] }=\beta _{N}(X)$ in the full shift over $\mathbf{A}_{X}^{\left[ N\right] }.$

Let $\mathcal{F}$ be a set of words over $\mathbf{A}$, and $X_{\mathcal{F}} $ the set of points $x\in \mathbf{A}^{\mathbb{Z}}$ such that no word of $\mathcal{F}$ occurs in $x.$ The set $\mathcal{F} $ is called a set of \emph{forbidden words} for $X.$ A shift of finite type is a shift space of the form $X_{\mathcal{F}}$ with $\left\vert \mathcal{F} \right\vert <\infty .$ A shift of finite type is $M$-step (or has memory $M$) if it can be described by a collection of forbidden blocks, all of which have length $M+1.$

Let $G=(V,E)$ be a directed graph with a finite set of vertices $V$ and a finite set of directed edges $E.$ For an edge $e\in E$ we denote by $\mathfrak{i}(e)$ and $\tau (e)$ the starting and terminating vertices of $e$ in $G.$ A bi-infinite path in $G$ is a sequence $(e_{i})_{i=-\infty }^{\infty}\subseteq E$ of edges such that for all integers $i$, $\tau
(e_{i})=\mathfrak{i}(e_{i+1})$. We also deal with finite paths, which are finite sequences $(e_{i})_{ i=1}^{\ell}\subseteq E$ of some length $\ell$ such that $\tau (e_{i})=\mathfrak{i} (e_{i+1})$ for $i=1,...,\ell-1$. Such a path is said to start at the vertex $\mathfrak{i} (e_{1})$ and end at the vertex $\tau (e_{\ell})$. A cycle is a finite path that starts and ends at the same vertex. A cycle $(e_{i})_{i=1}^{\ell}$ is simple if the vertices $\tau (e_{1}),...,\tau (e_{\ell})$ are distinct. For vertices $v,w\in V,$ let $A_{vw}$ denote the number of edges in $G$ with initial vertex $v$ and terminal vertex $w.$ The \textit{ adjacency matrix} of $G$ is defined as $A(G)=[A_{vw}].$ The \textit{edge shift} $X_{G}$ is the shift space over the alphabet $\mathbf{A}=E$ given by $X_{G}=X_{A(G)}=\{(\xi _{i})_{i\in \mathbb{Z}}\in E^{\mathbb{Z}}$ $|$ $\tau (\xi _{i})=\mathfrak{i}(\xi _{i+1})$ for all $i\in \mathbb{Z}\}.$ For a graph $G,$ the associated edge shift $X_{G}$\ is a 1-step shift of finite type, (\cite{LM1} Proposition 2.2.6). If $X$ is an $M$-step shift of finite type, then there is a graph $G$ such that $X^{[M+1]}=X_{G},$ (\cite{LM1} Theorem 2.3.2).

 A \textit{labeled directed graph} $\mathcal{G}$ is a pair $(G,\mathcal{L}),$ where $G$ is a directed graph with edge set $E,$ together with a labeling function $\mathcal{L} :E\rightarrow \mathbf{A}$ to a finite alphabet $\mathbf{A}.$ We call $G$ the underlying graph of $\mathcal{G}.$ For a labeled
graph $\mathcal{G}=(G,\mathcal{L} ),$ we define the label of a path $\pi=e_{1}\cdots e_{n}$ to be $\mathcal{L} \left( \pi \right) =\mathcal{L}\left( e_{1}\right) \cdots \mathcal{L} \left( e_{n}\right) ,$ and for a bi-infinite walk $\xi =\cdots e_{-1}.e_{0}e_{1}\cdots ,$ we have $\mathcal{L}_{\infty }\left( \xi \right) =\cdots \mathcal{L} \left( e_{-1}\right).\mathcal{L} \left( e_{0}\right) \mathcal{L} \left( e_{1}\right) \cdots .$ \\For a labeled graph $\mathcal{G}=(G,\mathcal{L})$, we define the \emph{labeled edge shift} $X_{\mathcal{G}}$ to be $X_{\mathcal{G}}
=\{{\mathcal{L} _{\infty }(\xi )\,\vert \, \xi \in X_{G}\}=\mathcal{L} _{\infty}}(X_{G}).$  A labeled graph $\mathcal{G}$ is said to be a \textit{presentation} of a shift space $X$ if $X=X_{\mathcal{G}}.$ If $w\in B\left( X_{\mathcal{G}}\right) ,$ we will say that a path $\pi $ in $G$ is a \textit{presentation} of $w$ if $\mathcal{L} (\pi )=w.$
If $x\in X_{\mathcal{G}},$ we will say that a bi-infinite walk $\xi $ in $X_{G}$ is a \textit{presentation} of $x$ if $\mathcal{L} (\xi )=x.$

 A shift space $X$ is said to be \textit{sofic} if there exists
some finite labeled graph $\mathcal{G}$ which is a presentation of $X$\ - that is, $X=X_{\mathcal{G}}.$ A labeled graph is called \textit{right resolving} if, for each vertex $v,$ the edges starting at $v$\ carry different labels. Every sofic shift has a right-resolving presentation, (\cite{LM1}, Theorem 3.3.2). \emph{A minimal right-resolving presentation }of a sofic
shift $X$\ is a labeled graph $\mathcal{G}$ such that $X=X_{\mathcal{G}}$\ and $\mathcal{G}$ has the fewest vertices among all right resolving presentations of $X.$

\section{Topological Entropy, Independence Entropy and Supremum independence entropy} \label{sec3}
Topological Entropy may be viewed as a measure of "information capacity" of a shift space - or of its ability to transmit messages.
\begin{definition}[Topological Entropy]
 Let $X$ be a shift space. The topological entropy of $X$ is defined by
\begin{equation*}
h(X)=\underset{n\rightarrow \infty }{\lim }\frac{1}{n}\ln \left\vert
B_{n}(X)\right\vert.
\end{equation*}
\end{definition}
The number of $n$-blocks appearing in sequences from $X$ gives us some
measure of complexity. The definition of entropy given above roughly
captures the growth rate of this complexity as $n$ increases. The limit exists (see for instance Lind and Marcus (\cite{LM1}, Prop. 4.1.8)), so the entropy is always well-defined. 
Topological entropy turns out to be invariant under conjugacy, and more generally:
\begin{proposition}[\cite{LM1}, Proposition 4.1.9]
If $Y$ is a factor of $X,$ then $h(Y)\leq h(X).$ In particular, if $X\simeq Y$ then $h(X)=h(Y).$
\end{proposition}

In order to say something about the entropy of sofic shifts, we need the following proposition, which connects a presentation of a sofic shift as a labeled graph with its entropy.
\begin{proposition}[\cite{LM1}, Proposition 4.1.13]\label{thm113}
Let $\mathcal{G}=(G,\mathcal{L} )$ be a right resolving labeled graph. Then $h(X_{\mathcal{G}})=h(X_{G}).$
\end{proposition}
 A non-negative matrix $A$ is\emph{\ irreducible }if for each ordered pair of indices $I,J,$ there exists some $n\geq 0$ such that $A_{IJ}^{n}>0.$ A graph $G$ is\emph{\ irreducible }if for every pair of
vertices $v,w\in V,$ there exists a path beginning at $v$ and terminating at $w.$
The largest eigenvalue of $A$ is called the \emph{Perron} \emph{eigenvalue}, and denoted by $\lambda _{A}$.

\begin{theorem}[\cite{LM1}, Theorem 4.3.1]\label{thm 431}
 If $G$  is an irreducible graph and $A$ is the adjacency matrix of $G$, then $h(X_{G})=\ln \lambda _{A}.$
\end{theorem}
\begin{theorem}[\cite{LM1}, Theorem 4.3.3]
 Let $\emph{X}$ be an irreducible sofic shift and $\mathcal{G}=(G,\mathcal{L} )$ be an irreducible right resolving presentation of ${X}$. Then $h(X)=\ln \lambda _{A},$ where $A$ is the adjacency matrix of $G$.
\end{theorem}
 Since entropy measures exponential growth, it seems natural that sub-shifts of $X$ with an entropy less than $X$ can be disregarded in calculating the entropy of $X$ as they only contribute sub-exponentially. This is the basis of the next proposition.
\begin{theorem}[\cite{LM1}, Theorem 4.4.4]\label{444}
 Let $G$ be a finite graph and let $T$ be the set of irreducible sub-graphs of $G.$ Then,
\begin{equation*}
h(X_{G})=\underset{H\in T}{\max }\text{ }h(X_{H}).
\end{equation*}
\end{theorem}
Louidor, Marcus and Pavlov introduce (\cite{LMP13}, and see also \cite{PCM04}) the concept of independence entropy for symbolic dynamical systems.
\begin{definition}[Fillings set]
 For a finite alphabet $\mathbf{A}.$ Let $\hat{\mathbf{A}}$
denote the set of all non-empty subsets of $\mathbf{A}.$ Let $X$ be a $\mathbb{Z}$-shift space over $\mathbf{A}.$ Let $\mathbf{A}^*$ denote the set of all finite words over $\mathbf{A}.$ For $\hat{x}\in {\hat{({\mathbf{A}})}}^*$ define the set of fillings of $\hat{x},$ denoted $\Phi(\hat{x}),$ by
\begin{equation*}
\Phi(\hat{x})=\{x\in{\mathbf{A}^*|\:\vert x \vert=\vert\hat{x}\vert,     \: \text{for all } i, x_{i}\in{\hat{x}_{i}}\}}.
\end{equation*}
\end{definition}
For convenience, we extend the domain of $\Phi$ to include bi-infinite words.
\begin{definition}
If $\hat{x} \in {\hat{\mathbf{A}}^\mathbb{Z}}$ is a bi-infinite word, $\Phi(\hat{x})$ is given by
\begin{equation*}
\Phi(\hat{x})=\{x\in{\hat{\mathbf{A}}^\mathbb{Z}}\:|\: \text{for all }i, x_{i}\in{\hat{x}_{i}}\}.
\end{equation*}
\end{definition}
 \begin{definition}[The multi-choice shift space]
 Let $X$ be a shift space. We define the multi-choice shift space corresponding to $X,$ denoted $\hat{X},$ by
 \begin{equation*}
 \hat{X}=\{\hat{x}\in{\hat{\mathbf{A}}^\mathbb{Z}}\:|\: \Phi(\hat{x})\subseteq X\} .
 \end{equation*}
\end{definition}
\begin{definition}[Independence entropy]\label{def:1}
Let $X$ be a shift space. The independence entropy denoted $h_{ind}(X)$, is defined by
\begin{equation*}
h_{ind}(X)=\lim_{m\rightarrow\infty}\frac{\max  \{\ln \vert\Phi(\hat{w})\vert\:|\:\hat{w}\in B_{m}(\hat{X})\}}{m}.
\end{equation*}
\end{definition}
The limit exists, (see \cite{LMP13} pages 302-303).
 \begin{example}
 Let $X$ be the golden mean shift. $X$ consists of all
bi-infinite \{0,1\} sequences which do not contain consecutive $1'$s. For any block $\hat{w} \in \hat{X}$, we have $|\Phi(\hat{w})|=2^{i},$ where $i$ is the number of occurrences of the symbol \{0,1\} in $\hat{w}.$ Therefore, the maximum value of $|\Phi(\hat{w})|$ for $\hat{w}$ of a fixed length $m$ will be achieved by maximizing the number of occurrences of \{0,1\} in $\hat{w}.$ For any $m,$ it is clear that there exists $\hat{w} \in  B_{m}(\hat{X})$ with $\lceil\frac{m}{2}\rceil$ occurrences of \{0,1\}, namely $\hat{w}=\{0,1\}\{0\}\{0,1\}\cdots$. It is also clear that there is no $\hat{w} \in B_{m}(\hat{X})$ with more occurrences of \{0,1\}, since such $\hat{w}$ would contain two consecutive \{0,1\}'s. Therefore, $h_{ind}(X)=\underset{m\to\infty}\lim\frac{\ln{2^{\lceil\frac{m}{2}\rceil}}}{m}=\frac{\ln2}{2}$.
 \end{example}
 For every $m\in \mathbb{N},$ and a block $\hat{z}(m)\in B_{m}(\hat X)$, we have $\Phi(\hat{z})\subseteq B_{m}(X).$ This yields the following theorem.
 \begin{theorem}[\cite{LMP13}, Theorem 3] 
 For any shift space $X$ over an alphabet $\mathbf{A}$, $h_{ind}(X)\leq h(X).$
 \end{theorem}
\begin{theorem}
 Let $X$ be a sofic shift space
over $\mathbf{A}.$ Let $X^{\left[ N\right] }$ be the $N$th higher block shift space of $X$. Then for any $N \geq 2,$ $h_{ind}(X^{\left[ N\right] })=0.$
\begin{proof}
 By contradiction, assume that $h_{ind}(X^{\left[ N\right]
})>0. $ Then there exist two elements $x,y\in X^{\left[ N\right] }$
differing in just one coordinate $x_{i}\neq y_{j}\Leftrightarrow i=i_{0},$ (see Lemma \ref{lem 5} below).
But this could not happen because of the overlap condition, (see page 12 in \cite{LM1}).
\end{proof}
\end{theorem}
\begin{corollary}
Let ${X}$ be a sofic shift space. Then
\begin{center}
$\inf h_{ind}\left( X\right) =\inf \left\{ h_{ind}\left( Y\right)
|Y\simeq X\text{ }\right\} =0.$
\end{center}
\end{corollary}
\begin{lemma}\label{lem 2}
If $X,Y$ are two shift spaces such that $X\subseteq
Y,$ then \begin{equation*}
h_{ind}(X)\leq h_{ind}(Y).
\end{equation*}
\begin{proof}
 $B_{n}\left( \widehat{X}\right) \subseteq B_{n}\left( 
\widehat{Y}\right) ,$ so the maximum of Definition \ref{def:1} is taken over a larger set for $\widehat{Y}. $
\end{proof}
\end{lemma}
The independence entropy is not invariant under topological conjugacy, but we can extract an invariant of conjugacy from it as follows. 
\begin{definition}
Consider the supremum of independence entropy over the collection of shift spaces that are isomorphic to the given one $X.$ We define the supremum independence entropy of $X$ by 
\begin{equation*}
\sup h_{ind}\left( X\right) =\sup \left\{ h_{ind}\left( Y\right) |\text{ }
Y\simeq X\text{ }\right\}.
\end{equation*}
\end{definition}
This is much less obvious to compute, but it is some number bounded above by the topological entropy, and it is a non-trivial invariant.
\section{Supremum independence entropy of one-dimensional sofic shift spaces} \label{sec4}
In this section we prove Theorem \ref{thm:main1}. Given a sofic shift $X$ and $\varepsilon > 0$, we produce a topological equivalent $Y$ with $h_{ind}(Y)> h(X)-\varepsilon$.\\
The idea is to find a collection of words $\Gamma=\Gamma_K$ of length $\eta=\eta_K$ with the following properties.\\
\begin{enumerate}
\item $\Gamma$ captures most of entropy, in the sense that for $k$ large \begin{center}
$\dfrac{\ln |\Gamma_K|}{\eta_K}>h(X)-2\varepsilon$.
\end{center} 
\item Two different words in $\Gamma$ never overlap.
\end{enumerate}
Having found such $\Gamma$ we define a new alphabet, \begin{center}
$\textbf{A}=\mathcal{A}\cup \{*\} \cup W_K,$
\end{center}
 where $W_K=\{\overline{w}\,|\, w\in \Gamma\}$ is a formal set of symbols representing the words in $\Gamma $. We embed $X$ into $Y\subseteq \textbf{A}^{\mathbb{Z}}$, by replacing every occurrence of a word $w \in \Gamma$ by the sequence $\overline{w}**\ldots*.$ This configuration  yields at least $h(X)-2\varepsilon$ in independence entropy for $K$ large, because the substitution of any $\overline{w} \in W_K$ gives rise to a legal word in $Y$.
\begin{proof} [Proof of Theorem \ref{thm:main1}] 
 Let $X$ be a sofic shift space over an alphabet $\mathbf{A}$ with $h(X)>0.$ Since $h_{ind}(Y) \leq h(Y)$ for every topologically conjugate shift, $\sup h_{ind}\left(X\right) \leq h(X).$ Let $\varepsilon >0$ be given. We will establish our proof by constructing a shift space $Y$ which is topologically conjugate to $X,$ with $h_{ind}(Y)>h(X)-\varepsilon .$
 
   Let $\mathcal{G}=\left(G,\mathcal{L} \right) $ be a minimal right-resolving presentation so that $X=X_{\mathcal{G}}.$ By Theorem \ref{444} there exists an irreducible sub-graph $H$ of $G$ such that $h(X_{G})=h(X_{H}) ,$ and therefore there is a minimal right-resolving presentation $\mathcal{H}
=\left( H,\mathcal{L} \right) ,$ so that $X_{\mathcal{H}}\subseteq X_{\mathcal{G}}$ is an irreducible sofic shift space.
  Suppose that $H$ has $v$ vertices.  Let $\rho $ be the length of the longest path between any two vertices in $H.$ By Theorem \ref{thm 431} and Proposition \ref{thm113}, $h\left( X_{\mathcal{G}}
\right) =h(X_{G})=h(X_{H})=h(X_{\mathcal{H}})=\ln \lambda ,$ where $\lambda $ is the Perron eigenvalue of the adjacency matrix $A$ corresponding to the graph $G.$  By Proposition 4.2.1 in \cite{LM1}, there is an $\alpha >0$ such that $|B_{k}(X_{H})|>\alpha \lambda ^{k}$ for $k\geq 1.$  Any labelled path in $B_{k}(X_{\mathcal{H}})$ has at most $v$ presentations in $B_{k}(X_{H})$ since $\mathcal{H}$ is right-resolving. Hence
 \begin{equation}\label{eq4.1}
\left\vert B_{k}(X_{\mathcal{G}})\right\vert\geq \left\vert B_{k}(X_{\mathcal{H}})\right\vert \geq \frac{1}{v}\left\vert
B_{k}(X_{H})\right\vert \geq \frac{\alpha }{v}\lambda ^{k}.
\end{equation}
 Since 
$\frac{n}{\ln \left( n\right) }\rightarrow \infty $ as $n\rightarrow \infty $
we can choose $N_{0}$ such that for all $n>N_{0},$ there exists $k,$ so that
\begin{equation} \label{eqn:kn_consts}
\frac{\ln( \frac{v\rho^2}{\alpha}n)}{(1-\varepsilon)\ln\lambda} < k < \frac{\varepsilon n - 2 \rho}{2}. 
\end{equation}
 Now fix $n$ and $k$ that satisfy \ref{eqn:kn_consts}, and let $M\in B_{n}\left( X_{\mathcal{H}}\right) $. Since $k< n,$ $M$ has $n-k+1$
sub-blocks of length $k.$ From \ref{eqn:kn_consts} and \ref{eq4.1}, we have  $|B_{k}( X_{\mathcal{H}})|>\frac{\alpha }{v}\lambda ^{k}>n-k+1,$ and therefore there will be a block $C\in
B_{k}( X_{\mathcal{H}})$ that does not appear as a sub-block of $M.$ Since $ X_{\mathcal{H}}$ is an irreducible shift space, there exists a block $S$ with length $l\leq \rho $ such that $MSC\in B\left(  X_{\mathcal{H}}\right) .$\\

 Let us enumerate  $B_{n-2k-2\rho }( X_{\mathcal{H}})=\{F_j\}$. For every such $F_j$ fix, once and for all, a path $F_j'$ in $H$ representing it in the sense that $\mathcal{L}(F_j')=F_j$. Similarly fix  presentations $C'$ of $C$, $M'$ for $M$ and $S'$ for $S$.
 
  For every $F_j \in B_{n-2k-2\rho }( X_{\mathcal{H}})$, let us fix  paths $L'_j,R'_j$ of lengths $\ell_j,r_j \leq \rho$ respectively concatenating it with $C'$. Thus any path of the following form
\begin{equation*}
\cdots C'L'_jF'_jR'_jC' \cdots ,
\end{equation*} 
represents a legal word in $B(X_{\mathcal{H}})$.

For $s,t \leq \rho$, let
\begin{center}
$\Upsilon_{s,t}=\{ F_i \in B_{n-2k-2\rho }( X_{\mathcal{H}})\:\vert\: \ell_j=s,\:r_j=t\},$
\end{center}  
and choose $\ell,r \leq \rho$ such that $|\Upsilon_{\ell,r}|\geq|\Upsilon_{s,t}|$, for all $0\leq s,t\leq\rho.$
 By the Pigeon-hole Principle and \ref{eq4.1},\begin{equation}
|\Upsilon_{\ell,r}|\geq\left\vert\frac{B_{n-2k-2\rho }(X_{\mathcal{H}})}{\rho^{2}}\right\vert > \frac{\alpha }{v\rho^{2}}\lambda
^{n-2k-2\rho }.
\end{equation}

For $K\in \mathbb{Z}$, set $\eta_K=(n-k-2\rho +\ell + r)K +n+k+l$. Let $\Gamma _{K}\subseteq B_{\eta_K}(X_{\mathcal{H}})$ be the collection of all blocks of the form:\begin{equation*}
 \Gamma _{K}=\{w\in B_{\eta_K}(X_{\mathcal{H}})\: |\:w=MSCL_1F_{1}R_1CL_2F_{2}R_2C\cdots
CL_KF_{K}R_KC\},
\end{equation*}
where, for all $i:1\leq i \leq K$, $F_i \in \Upsilon_{\ell,r}$. 

Note that our prior explicit choice for paths representing such words, ensures that every choice of $\{F_i \in \Upsilon_{\ell,r} \:|\: 1\leq i\leq K \} \subseteq \Upsilon_{\ell,r}^K$ gives a distinct legal word in $\Gamma_K$. Yielding a bijection \begin{center}
$\Upsilon_{\ell,r}^K\eqsim \Gamma_K \subseteq B_{\eta_K}(X).$
\end{center}

Hence, using \ref{eqn:kn_consts}, the collection  $\Gamma _{K}$ of blocks $w$ with the required
form has cardinality
\begin{equation}
\left\vert \Gamma _{K}\right\vert > \left( \frac{\alpha }{v\rho^{2}}\lambda
^{n-2k-2\rho }\right)^K >\left( \frac{\alpha }{v\rho^2}\lambda ^{\left(1-\varepsilon \right) n}\right) ^{K}.
\end{equation}

 Consider each block $w \in \Gamma_{K}$ as a symbol, say $\overline{w}.$ Let  $W_{K}=\{\overline{w} | w \in \Gamma_{K}\}$ be the set of symbols representing the words in $\Gamma_{K}$. Define a new alphabet $\mathbf{A}_{K}$ as the disjoint union: $\mathbf{A}_{K}=\mathbf{A}\cup \{\ast \}\cup W _{K}.$ Define a map $\Psi :X\rightarrow \left( \mathbf{A}_{K}\right) ^{\mathbb{Z}}$ as follows.
 \begin{eqnarray*}
\Psi (x)_{i} =  \left\{ \begin{array}{ll} {\overline{w}} , & \text{If  $w$  occurs in $x$ at $i$ }, \\
\ast, &  \text{If $w$ occurs in $x$  at $i-j$ for some $1\leq
j\leq (\eta_K-k-1)$ } ,\\
x_i,& \text{Otherwise.}\end{array} \right. \\
\end{eqnarray*}
 We will show that the map $\Psi $ is well defined. Let $w,w' \in \Gamma_{K}$ and assume these two blocks overlap in the sense that $w_{[q+1,\eta_K]}=w'_{[1,\eta_K-q]}.$ If $l+k \leq q\leq \eta_K -k-1$ then a comparison of the two blocks immediately shows that $M$ contains $C$ as a sub-block and yields a contradiction.
If $1\leq q \leq l+k-1.$ Then the block $D=MSC$ is $q-$periodic as is obvious from the comparison of the overlapping words below (Figure 1). Setting $|D|=n+l+k=(s+2)q+i$ for some integers $s,i$, $q$-periodicity implies that  $D_{[i+1,i+k]}=C$  and this occurrence of $C$ is contained in $M$ which is a contradiction. 
 \begin{figure}[h]\label{figure}
\begin{tabular}{|c|l|c|c|c|}
$d_1d_2\cdots d_q$ & $d_{q+1}\cdots d_{2q}$ &$\cdots$  &$d_{(s+1)q+1} \cdots d_{n+i}$ & $d_{n+i+1}\cdots  d_{n+l+k}$\\
 &$d_1d_2\cdots d_q$ & $\cdots $ &$\underset{d_1d_2\cdots d_q}{\underbrace{ d_{sq+1}\cdots  d_{(s+1)q}}}$  &$\underset{d_1d_2\cdots d_i \cdots}{\underbrace{ d_{(s+1)q+1} \cdots \cdots}}$
\end{tabular}
\caption{The comparison of the two overlapping words.}
\end{figure}
  
  As a sliding block code, the map $\Psi $ is  continuous and commutes with the shift map. By definition it is injective. As $X$ is compact and $\left( \mathbf{A}_{K}\right) ^{\mathbb{Z}}$ is Hausdorff, $X$ and $Y_{K}=\Psi \left( X\right) $ are topologically conjugate. 

 Let $\widehat{Y}_{K}$ be the multi-choice shift
space corresponding to $Y_{K}$ over the alphabet $\widehat{\mathbf{A}}_{K}.$ Let 
$\widehat{w}_{K}\in B(\widehat{Y}_{K}),$ be given by:\begin{equation*}
\widehat{w}_{K}  =  {W_{K}\underset{(\eta_K -k-1) 
\text{-times}}{\underbrace{\{\ast \}\{\ast\}\ldots \{\ast \}}}}.
\end{equation*}
 Hence, $|\Phi (\widehat{w}_{K})|=\left\vert W_{K}\vert\vert\{\ast \}\vert \ldots
\vert\{\ast \}\right\vert =\left\vert \Gamma _{K}\right\vert. $\\
Therefore,
\begin{eqnarray*}
\sup h_{ind}\left( X\right) &\geq &\underset{ K\rightarrow \infty }{\lim \sup }\: h_{ind}(Y_{K}) \\
&\geq &\limsup_{K\rightarrow \infty }\frac{\max \text{ }\{\ln \text{ }
|\Phi (\widehat{u})|\:|\text{ }\widehat{u}\in B_{\eta_K}(\widehat{Y}
_{K})\}}{\eta_K} \\
&\geq &\underset{K\rightarrow \infty }{\limsup }\frac{\ln |\Phi (%
\widehat{w}_{K})|}{\eta_K}\\
&\geq & \underset{K\rightarrow \infty }{%
\lim \sup }\frac{\ln \left( \frac{\alpha }{v\rho^2}\lambda ^{\left( 1-\varepsilon
\right) n}\right) ^{K}}{(n-k-2\rho +\ell + r)K +n+k+l} \\
&\geq & \underset{K\rightarrow \infty }{%
\lim \sup  }\frac{\ln \left( \frac{\alpha }{v\rho^2}\lambda ^{\left( 1-\varepsilon
\right) n}\right) ^{K}}{(n-k)K+n+k+l} \\
&\geq&\left( 1-\varepsilon \right) \ln \lambda =h(X_{\mathcal{H}})-\varepsilon \ln \lambda =h(X)-\varepsilon \ln \lambda.
\end{eqnarray*}
Since $\varepsilon >0$ was arbitrary, the theorem is proved.
\end{proof}

In our definition of $\sup h_{ind}(X)$,  a natural question that appears  is whether this supremum is achieved by any sofic shift. For full shift spaces the answer is clearly positive. One of referees pointed out to us that the answer for general sofic shifts is no. Indeed, if $X$ is irreducible sofic the  independence entropy is always equal to $\ln (n)/k$ for some integers $n$ and $k$ (\cite{LMP13}, Theorem 2), while there are irreducible sofic shifts whose topological entropy is not of this form, for example the topological entropy for golden mean shift is equal $\ln \frac{1+\sqrt{5}}{2}$. Despite this, it remains interesting to find more examples where the equality holds.
 \begin{problem}
Find interesting examples or even characterize the shift spaces $X$ for which $h(X)=h_{ind}(X)=\sup h_{ind}(X)$.
 \end{problem}
 Note that by \cite{meypav} Theorem 1.1, the topological entropy of such a shift space $X$ will be equal to that of all of its axial powers. Along the same lines if $X$  is any sofic shift, Theorem \ref{thm:main1} provides a topological conjugacy shift $Y$ whose entropy is arbitrary close to that of its axial powers.
 
\section{Supremum independence entropy of general one-dimensional shift spaces} \label{sec5}
We will use the concept of asymptotic pairs to prove Theorem \ref{thm:main2}.
\begin{definition}
 $x,y\in X$ are called an asymptotic pair if they differ on finitely many letters. We call an asymptotic pair non-trivial if they are unequal.
\end{definition}
In \cite{Schmidt95} Proposition 2.1, Klaus Schmidt showed that  a shift space of finite type with positive entropy always has non-trivial asymptotic pairs. We now show that for independence entropy, the analogue of this result  holds in much greater generality.
 \begin{lemma}\label{lem 5}
 Let $X$ be a shift space over the alphabet $\mathbf{A.}$ If $h_{ind}(X)>0,$ then $X$ has non-trivial asymptotic pairs.
  \begin{proof}
 Suppose that $h_{ind}(X)>0.$ Assume that the alphabet $\mathbf{A}$ has some linear order. By $\widehat{X}$ we denote the
multi-choice shift space of $X$. Let $\widehat{z}\left( n\right) \in B_{n}\left( \widehat{X}\right) $ be the word such that:\begin{center} $\left\vert \Phi \left( \widehat{z}\left( n\right) \right) \right\vert =\max \left\{ \left\vert \Phi \left( \widehat{w}\right) \right\vert\:\vert \:\widehat{w}\in B_{n}\left( \widehat{X}\right) \right\} .$\end{center}

Let $\mathcal{R} \subset 2^{\widehat{\mathbf{A}}}$ be the collection of all subsets of $\widehat{\mathbf{A}}$ which contain more than one element. We have  $\mathcal{R}\cap B_1(\widehat{X})\neq\emptyset$, since  otherwise $\left\vert \Phi ( \widehat{w}) \right\vert = \vert\widehat{w}_0\vert.\vert \widehat{w}_1\vert...\vert \widehat{w}_{n-1}\vert=1$ for all $ n \geq 0$ and $ \widehat{w}\in B_n(\widehat{X})$, implying $h_{ind}(X)=0.$ Hence there exists $\widehat{z} \in \widehat{X}$ with  $\widehat{z}_0 \in \mathcal{R}$. Being an allowed word $\widehat{z}\left(n\right) $ appears in some $\widehat{z}$ in $\widehat{X,}$ say $\widehat{z}_{\left[ 0,n-1\right]}=\widehat{z}\left( n\right) .$ Let $i_{0}$  be the lowest non-negative index such that $\widehat{z}_{i_{0}}\in \mathcal{R} .$ Let $x,y\in \Phi \left( \widehat{z}\right) \subseteq X$ be points obtained from $\widehat{z}$ by the following construction:
\begin{eqnarray*}
x_i & = &  {\text{The first element of }} \hat{z}_i,{\text{ for all $i$ }}. \\
y_i & = & \left\{ \begin{array}{ll} {\text{The second element of }} \hat{z}_i, & i=i_0, \\
x_i, & i \ne i_0 .\end{array} \right. \\
\end{eqnarray*}
Hence $x\neq y$ but there is exactly one index for $x_{i}\neq y_{i}.$ So $(x,y)$ is an asymptotic pair.
 \end{proof}
 \end{lemma}
 A generalization of Schmidt's Theorem (\cite{Schmidt95} Proposition 2.1), to sofic systems now follows from Lemma \ref{lem 5} and Theorem \ref{thm:main1}.
 \begin{corollary}
  A sofic shift space $X$  with positive entropy always has non-trivial asymptotic pairs.
 \end{corollary}
 \begin{theorem}[\cite{Meyerovitch19} Section 5]
There is a shift space that has positive entropy but no asymptotic pairs.
\end{theorem}
\begin{proof}[Proof of Theorem \ref{thm:main2}]
Let $X$ be the shift space constructed in (\cite{Meyerovitch19} section 5), with\emph{\ }$h(X)>0,$ and no asymptotic pairs. Since the
existence of asymptotic pairs is invariant, any topologically conjugate $Y\backsimeq X,$ will have no asymptotic pairs. By Lemma \ref{lem 5}, $h_{ind}(Y)=0,$ for all such $Y.$ Hence $\sup h_{ind}(X)=0.$
\end{proof}
Thus the equality $h(X)=\sup h_{ind}(X)$, which holds for sofic shifts does not hold in general. It is interesting  to attain a better  understanding exactly where the line passes. Two natural test cases that where suggested to us by the referees are sub-shifts with \emph{specification property} (specification property states that there exists $N$ so that for any words $v,w \in B(X)$, there exists $u$ with length $N$ so that $vuw \in B(X)$),  and \emph{Coded sub-shifts.}  \\ 
Concerning sub-shifts that have the specification property. It was noted by the referee that the proof of our Theorem \ref{thm:main1} works almost verbatim in this case.\\

\textbf{Acknowledgements 1.}
This work is a part of the Master's thesis of the author,
conducted at Ben-Gurion University of the Negev under the supervision of Yair Glasner and Tom Meyerovitch. The author thanks his advisers for their patience, support and knowledge. The author also extends his sincere thanks to Ben-Gurion University for financial support and research funding.\\
The author was partially supported by ISF grant 2919/19 and ISF grant
1052/18.\\

\textbf{Acknowledgements 2.}
The author would like to thank the two referees for reading the first and second drafts and making comments and  suggesting constructive changes.

\newpage
\bibliographystyle{alpha}

\bibliography{bashir}

\begin{thebibliography}{PCM06}

\bibitem[LM95]{LM1}
Douglas Lind and Brian Marcus.
\newblock {\em An introduction to symbolic dynamics and coding}.
\newblock Cambridge university press, 1995.

\bibitem[LMP13]{LMP13}
Erez Louidor, Brian Marcus, and Ronnie Pavlov.
\newblock Independence entropy of {$\Bbb{Z}^d$}-shift spaces.
\newblock {\em Acta Appl. Math.}, 126:297--317, 2013.

\bibitem[Mey19]{Meyerovitch19}
Tom Meyerovitch.
\newblock Pseudo-orbit tracing and algebraic actions of countable amenable
  groups.
\newblock {\em Ergodic Theory Dynam. Systems}, 39(9):2570--2591, 2019.

\bibitem[MP14]{meypav}
Tom Meyerovitch and Ronnie Pavlov.
\newblock On independence and entropy for high-dimensional isotropic subshifts.
\newblock {\em Proceedings of the London Mathematical Society},
  109(4):921--945, 2014.

\bibitem[PCM06]{PCM04}
Tze-Lei Poo, Panu Chaichanavong, and Brian Marcus.
\newblock Tradeoff functions for constrained systems with unconstrained
  positions.
\newblock {\em IEEE transactions on information theory}, 52:1425–1449, 2006.

\bibitem[Sch95]{Schmidt95}
Klaus Schmidt.
\newblock The cohomology of higher-dimensional shifts of finite type.
\newblock {\em Pacific J. Math.}, 170(1):237--269, 1995.

\end{thebibliography}

%
%
%
\bigskip
\noindent {\sc Bashir Abu Khalil.} Department of Mathematics.
Ben-Gurion University of the Negev.
P.O.B. 653,
Be'er Sheva 84105,
Israel.\\
{\tt bashira@post.bgu.ac.il}\\
{\tt mathbash@gmail.com}

\end{document}